\documentclass[10pt]{amsart}
\usepackage{amsmath, url, cite, amsthm, amscd, amsfonts, amssymb, graphicx, color,fancyhdr,nomencl,colortbl}
\usepackage[colorlinks]{hyperref}
\usepackage{longtable,latexsym,epsfig,caption,subcaption,microtype,textcomp,stmaryrd,listings,floatrow}
\hypersetup{linkcolor=red,citecolor=blue}
\usepackage{fancyhdr,nomencl,amsmath,array,lineno,amsfonts,titlesec,amsthm,graphics,setspace,caption,epsf,cite}
\usepackage[T1]{fontenc}
\usepackage[sc,osf]{mathpazo}
\usepackage[euler-digits,small]{eulervm}
\AtBeginDocument{\renewcommand{\hbar}{\hslash}}
\usepackage{epstopdf}

\titlelabel{\thetitle. }

\numberwithin{equation}{section}
\newcommand{\ra}[1]{\mathrm{rank}({#1})}
\newcommand{\tr}[1]{\mathrm{Tr}({#1})}
\newcommand{\ind}[1]{\mathrm{ind}({#1})}

 \newtheorem{thm}{Theorem}[section]
 
 \newtheorem{lem}{Lemma}[section]
 
 \newtheorem{exm}{Example}[section]

\begin{document}
\centerline {\LARGE{\bf Highly Efficient Computation of 
Generalized Inverse of a Matrix}}
\centerline{}
\centerline{}
\centerline{
V.Y. Pan $^{a,c,}$\footnote{Email: victor.pan@lehman.cuny.edu; 
http://comet.lehman.cuny.edu/vpan/ \\
supported by NSF Grant CCF--1116736
and PSC CUNY Award  68862--00 46},
F. Soleymani $^{b,}$\footnote{Corresponding author. Email: fazl\_soley\_bsb@yahoo.com},
L. Zhao $^{c,}$\footnote{Email: lzhao1@gc.cuny.edu}
}
{\footnotesize
\centerline{$^{a}$ Department of Mathematics and Computer Science, Lehman College of CUNY, Bronx, NY 10468, USA}
\centerline{$^{b}$ Instituto Universitario de Matem\'{a}tica Multidisciplinar, Universitat Polit\`{e}cnica de Val\`{e}ncia, 46022 Val\`{e}ncia, Spain}
\centerline{$^{c}$ Departments of Mathematics and Computer Science, The Graduate Center of CUNY, New York, NY 10036 USA}
}

\medskip
\noindent {\bf Abstract.}
We propose a hyperpower iteration for numerical computation of the outer generalized inverse of a matrix
which achieves the 18th order of convergence by using only seven matrix multiplication per iteration loop. This is the record high efficiency for that computational task.
The algorithm has a relatively mild numerical instability, and we stabilize it at the price of adding 
one extra matrix multiplication per iteration loop. This imlplies an efficiency index that  significantly exceeds  the known record for numerically stable iterations for this task.
Our numerical tests cover a variety of examples such as Drazin case, rectangular case, and preconditioning of linear systems. The test results are in good accordance 
with our formal study and indicate that our algorithms can be of interest for the user. 
\medskip

\smallskip

\centerline{}

\section{Our Subject, Motivation, Related Works, and Our Progress}
\subsection{Generalized inverses: some applications}
It has been stated already by Forsythe et al. \cite[p. 31]{Forsythe} that in the great majority of practical computational problems, it is unnecessary and inadvisable to actually compute the inverse of a nonsingular matrix.
This general rule still remains essentially true for modern matrix computations (see, e.g., \cite[pages 39 and 180]{Stewart}).
In contrast the computation or approximation of {\em generalized inverses} is required in some important  matrix computations (cf., e.g., \cite{Nashed}).
For example, generalized inverses are used for preconditioning large scale linear systems of equations \cite{Benzi,Chow} and \cite[pp. 171-208]{Campbell} and
 for updating the regression estimates based on the addition or deletion of the data
in linear regression analysis \cite[pp. 253-294]{Ben-Israel}. 
Furthermore  the computation of the so-called zero initial state system
inverses for linear time-invariant state-space systems is essentially equivalent to determining generalized inverses of the matrices of the associated transfer functions.

There exists a generalized inverse of an arbitrary matrix, and it turns into a unique inverse when the matrix is nonsingular, but we must  compute generalized inverses in order to  deal with rectangular and rank deficient matrices \cite{Soleimani,Stanimirovic2013-2}. Some generalized inverses can be defined in any mathematical structure that involves associative multiplication, i.e., in a semigroup \cite[chapter 1]{Wang}.

A system $Ax=b$ of linear equations has a solution 
if and only if the vector $A^{\dag}b$ is a solution,
and if so, 
then  all solutions are given by the following expression:
\begin{equation}
x=A^{\dag}b + [I-A^{\dag}A]w,
\end{equation}
where we can choose  an arbitrary vector  $w$ and any generalized inverse $A^{\dag}$. 

\subsection{Outer generalized inverse}
Hereafter $\mathbb{C}^{m\times n}$  denotes the set of all complex $m\times n$ matrices, $\mathbb{C}_r^{m\times n}$ denotes the set of all complex $m\times n$ matrices of rank $r$, $I_m$ denotes the $m\times m$ identity matrix, and we drop the subscript if the dimension $m$ is not important or is clear from context. Furthermore $A^*$, $R(A)$, and $N(A)$ denote the conjugate (Hermitian) transpose, the Range, and the Null Space of a matrix $A$ $\in$ $\mathbb{C}^{m\times n}$, respectively. 

For $A$ $\in$ $\mathbb{C}^{m\times n}$, outer generalized inverses or $\{2\}$-inverses are defined 
 \cite{Ben-Israel} by
\begin{equation}\label{1}
A\{2\} = \{X \in \mathbb{C}^{n \times m} : XAX=X\}.
\end{equation}
For two fixed subspaces $S\in \mathbb{C}^{n}$ and  $T\in \mathbb{C}^{m}$, define
the generalized inverse $ A_{T,S}^{(2)} \in A\{2\}$ of a complex matrix $A \in \mathbb{C}^{m \times n}$ as the matrix $X \in \mathbb{C}^{n \times m}$ such that $R(X)=T$ and $N(X)=S$.

\begin{lem}
 Let a matrix $A\in\mathbb{C}^{m \times n}$ have rank $r$ and let $T$ and $S$ be subspaces of $\mathbb{C}^{n}$ and $\mathbb{C}^{m}$, respectively, with $dimT=dim S^{\perp}=t\leq r$. Then
 $A$ has a \{2\}$\textendash$inverse $X$ such that $R(X)=T$ and $N(X)=S$ if and only if
\begin{equation}\label{2}
AT \bigoplus S =\mathbb{C}^{m},
\end{equation}
in which case $X$ is unique and is denoted by $A_{T,S}^{(2)}$ (see, e.g., \cite{Wei3}).
\end{lem}

The traditional generalized inverses, e.g., the pseudo-inverse $A^\dag$ (a.k.a. Moore-Penrose inverse), the weighted Moore-Penrose inverses $A_{MN}^\dag$ (where $M$ and $N$ are two square Hermitian positive definite
matrices), the Drazin-inverse  $A^D$, the group inverse $A^\#$, the Bott-Duffin inverse $A_{L}^{-1}$ \cite{Sheng2012}, the generalized Bott-Duffin inverse $A_{L}^{\dag}$, and so on,
 each of special interest in matrix theory,
 are special cases of the generalized outer inverse $X=A_{T,S}^{(2)}$.
\subsection{The known iterative algorithms for generalized inverses}
A number of  direct
and iterative methods has been proposed and implemented  for the computation of generalized inverses (e.g., see \cite{Pan2006,Schreiber}). Here we consider iterative methods.
 They approximate generalized inverse preconditioners,
can be implemented efficiently
in parallel architecture, converge particularly fast in some special cases (see, e.g., \cite{LiuLAA}),
and 
 compute various generalized inverses by using the same procedure for different input matrices, while direct methods usually require much more computer time and space in order to achieve such results.

Perhaps the most general and well-known scheme in this category is the
following hyperpower iterative family of matrix methods \cite{Climent,SoleymaniLAA,Sticrel},
\begin{equation}\label{Hyperpower1}
X_{k+1}=X_k(I+R_k+\cdots+R_k^{p-1})=X_k\sum\limits _{i=0}^{p-1}R_k^i,\ \ \ R_k=I-AX_k,\qquad k\geq0.
\end{equation}
Straightforward implementation of the iteration (\ref{Hyperpower1}) of  order $p$ involves $p$ matrix-matrix products.
For  $p=2$ it
turns into the Newton-Schulz-Hotelling matrix iteration (SM), originated in \cite{Hotelling1943,Schulz}:
\begin{equation}\label{schulz}
X_{k+1}=X_k(2I-AX_k),
\end{equation}
and for $p=3$ into the cubically convergent method of Chebyshev-Sen-Prabhu (CM) \cite{Sen}:
\begin{equation}\label{chebyshev}
X_{k+1}=X_k(3I-AX_k(3I-AX_k)).
\end{equation}

The paper \cite{SoleymaniNUMA} proposed the following
 seventh-order factorization (FM)
for computing outer generalized inverse with prescribed range and null space 
assuming
an appropriate initial matrix 
$X_0$ (see Section 4 for its choices):
\begin{equation}\label{SoleymaniN}
\left\{ \begin{array}{l}
\psi_k=I-AX_k,\\
\zeta_k=I+\psi_k+\psi_k^2,\\
\upsilon_k=\psi_k+\psi_k^4,\\
X_{k+1}=X_k(I+\upsilon_k\zeta_k).
\end{array}\right.
\end{equation}
Chen and Tan \cite{Chen} proposed computing $A_{T,S}^{(2)}$ by iterations based on splitting matrices.

For further background of iterative methods for computing generalized inverses, one may consult \cite[pp. 82-84]{Isaacson}, \cite[chapter 1]{NashedBook}, 
\cite{Nashed}, \cite{Ben-Israel}, \cite{Campbell},  \cite{Zheng}. Ben-Israel \cite{Ben-Israel2}, Pan \cite{Pan2010} and Sticrel \cite{Sticrel} have presented general introductions into iterative methods for computing $ A_{T,S}^{(2)}$. Recently such methods have  been studied extensively together with their applications (see, e.g., \cite{Codevico,Liu2013,Petkovic2}).

\subsection{Our results}
Our main results are two new  
 algorithms in the form (\ref{Hyperpower1}) for the generalized matrix inverse. 
They involve only 7 and 8 matrix-by-matrix products, respectively, 
and both of them achieve the convergence rate of 18. 
The efficiency index of the first algorithm (involving seven products) is record high,
but the algorithm is numerically unstable, although mildly.
Our second algorithm, using one extra matrix-by-matrix multiplication,  is numerically stable.
Its efficiency index is substantially higher than the previous record among 
numerically stable iterations for the same task.
Our numerical tests showed that our algorithms are quite competitive and in most cases superior to the known algorithms in terms of the CPU time involved.
All this should make our study theoretically and practically interesting.

\subsection{Organization of the  paper}

 In Section 2 we present our new algorithm. Its convergence and error analysis are the subjects of Section 3. In Section 4
we comment on the choice of the choice of 
an initial approximate inverse.
In Section 5
we  discuss its computational efficiency, while Section 6 is devoted to the analysis of its numerical stability. In Section 7 we present our second, numerically stable algorithm. Numerical tests, including the Drazin case, rectangular case, and preconditioning of large matrices, are covered  in Section 8. We
 measure the performance by the number of iteration loops, the mean CPU time, and the error bounds. In our tests we compare performance of our algorithm and
the known methods and show our
 improvement in terms of both computational
time
and accuracy. In Section 9 we present our brief concluding remarks and point out some further research directions.

\section{Our First Fast Algorithm}
It is well known that algorithm (\ref{schulz}) has polylogarithmic complexity
and is numerically stable and even self-correcting if the matrix $A$
is nonsingular, 
but otherwise is mildly unstable \cite{Soderstorm},
\cite{Pan1991}. Moreover it converges quite 
slowly
in the beginning. Namely, its initial convergence is linear, and  many iteration loops are
generally required in order to arrive at the final quadratic convergence \cite[pp. 259-287]{Higham2002}.
A natural remedy is 
provided by  higher order matrix methods
using fewer matrix-by-matrix multiplications, which are the cost dominant operations
in hyperpower iterations (\ref{Hyperpower1}).

Let $A\in\mathbb{C}^{m \times n}$,
 let $T$ and $S$ be subspaces of $\mathbb{C}^{n}$ and $\mathbb{C}^{m}$, respectively, with $dim T=dim S^{\perp}=t\leq r$, assume that $G\in\mathbb{C}^{n \times m}$ satisfies 
$R(G)\subseteq T$ and $N(G)\supseteq S$, and write $X_0 = \alpha G$, 
for a 
nonzero real scalar  $\alpha$ and a matrix $G$, 
both specified in Section 4.
Now define a  hyperpower iteration,
 for
$p=18$ and any $k = 0, 1, 2, \ldots$, by
\begin{equation}\label{hyper18}
X_{k+1}=X_k(I+R_k+R_k^2+\cdots+R_k^{17}), \qquad R_k=I-AX_k.
\end{equation}

The algorithm has the 18-th order of convergence and involves 18 matrix-by-matrix products
per iteration loop, but we are going to use fewer products.

Based on factorization (\ref{hyper18}), we obtain (HM)
\begin{equation}\label{hyper18-2}
X_{k+1}=X_k (R_k+I) \left(R_k^2-R_k+I\right) \left(R_k^2+R_k+I\right) \left(R_k^6-R_k^3+I\right) \left(R_k^6+R_k^3+I\right),
\end{equation}
and consequently
\begin{equation}\label{hyper18-3}
X_{k+1}=X_k (I+R_k) \left(I+R_k^2+R_k^4+R_k^6+R_k^8+R_k^{10}+R_k^{12}+R_k^{14}+R_k^{16}\right).
\end{equation}
Iterations (\ref{hyper18-2}) and (\ref{hyper18-3}) are clearly
superior to the original scheme (\ref{hyper18}),
but we will simplify them further.

Consider the following 
iterations,
\begin{equation}\label{hyper18-33}
\begin{split}
 I+R_k^2+R_k^4+R_k^6+R_k^8+R_k^{10}+R_k^{12}+R_k^{14}+R_k^{16}=&(
I+a_1R_k^2+a_2R_k^4+a_3R_k^6+R_k^8)\\
 &\times(
I+b_1R_k^2+b_2R_k^4+b_3R_k^6+R_k^8)+(\mu R_k^2+\psi R_k^4)
\end{split}
\end{equation}
where  
we write $a_3=b_3$ and
select
 seven nonzero real parameters
 $a_1,a_2,a_3,b_1,b_2,b_3,\mu,\psi$
from the following
 system of seven nonlinear equations:
\begin{equation}
\left\{ \begin{array}{l}
\mu + a_1 + b_1=1,\\[1mm]
a_2 + \psi + a_1 b_1 + b_2=1,\\[1mm]
2 a_3 + a_2 b_1 + a_1 b_2=1,\\[1mm]
2 + a_1 a_3 + a_3 b_1 + a_2 b_2=1,\\[1mm]
a_1 + a_2 a_3 + b_1 + a_3 b_2=1,\\[1mm]
a_2 + a_3^2 + b_2=1,\\[1mm]
2 a_3=1.
\end{array}\right.
\end{equation}
We obtain
\begin{equation}
a_1= \frac{5}{496} \left(31+\sqrt{93}\right),\quad a_2= \frac{1}{8} \left(3+\sqrt{93}\right),\quad a_3= \frac{1}{2},
\end{equation}
\begin{equation}
b_1= \frac{-5}{496} \left(\sqrt{93}-31\right),\quad b_2= \frac{1}{8} \left(3-\sqrt{93}\right),\quad
\mu= \frac{3}{8},\quad \psi= \frac{321}{1984}.
\end{equation}
Factorization (\ref{hyper18-3}) enables us to reduce the number of matrix-by-matrix multiplications to eight, but we are going to simplify this procedure further. We apply a similar strategy
and deduce the following factorization:
\begin{equation}\label{simplify1}
\begin{split}
1+a_1R_k^2+a_2R_k^4+a_3R_k^6+R_k^8=&(1+c_1R_k^2+R_k^4)(1+c_2R_k^2+R_k^4)+(c_3 R_k^2).
\end{split}
\end{equation}
By solving the  nonlinear system of algebraic equations
\begin{equation}
\left\{ \begin{array}{l}
c_1 + c_2 + c_3=a_1,\\[1mm]
2 + c_1 c_2=a_2,\\[1mm]
c_1 + c_2=a_3,
\end{array}\right.
\end{equation}
we obtain
\begin{equation}
c_1= \frac{1}{4} \left(\sqrt{27-2 \sqrt{93}}+1\right),\quad c_2= \frac{1}{4} \left(1-\sqrt{27-2 \sqrt{93}}\right),\quad c_3= \frac{1}{496} \left(5 \sqrt{93}-93\right).
\end{equation}
Furthermore write
\begin{equation}\label{simplify2}
\begin{split}
1+b_1R_k^2+b_2R_k^4+a_3R_k^6+R_k^8=&(1+c_1R_k^2+R_k^4)(1+c_2R_k^2+R_k^4)+(d_1 R_k^2+d_2 R_k^4)
\end{split}
\end{equation}
and by solving the nonlinear system of equations
\begin{equation}
\left\{ \begin{array}{l}
d_1+\frac{1}{2}=b_1,\\[1mm]
\frac{1}{8} \left(8 d_2+\sqrt{93}+3\right)=b_2,
\end{array}\right.
\end{equation}
deduce that
\begin{equation}
d_1=\frac{1}{496} \left(-93-5 \sqrt{93}\right),\quad d_2=-\frac{\sqrt{93}}{4}.
\end{equation}

Summarizing, we arrive at the following iterative method (PM) for computing generalized inverse:
\begin{equation}\label{PM}
\left\{ \begin{array}{l}
R_k= I-AX_k,\quad R_k^2 = R_kR_k,\quad R_k^4 = R_k^2R_k^2,\\[1mm]
M_k = (I + c_1 R_k^2 + R_k^4)(I+ c_2 R_k^2 + R_k^4),\\
T_k = M_k + c_3 R_k^2,\quad S_k = M_k + d_1 R_k^2 + d_2 R_k^4,\\
X_{k+1} = X_k((I + R_k)((T_kS_k) + \mu R_k^2 + \psi R_k^4)).
\end{array}\right.
\end{equation}
The iteration requires only seven matrix-by-matrix multiplications per loop,
and as we show next, the algorithm has convergence rate eighteen.

\section{Convergence and Error Analysis}

In this section we present convergence and error analysis of our algorithm (\ref{PM}).

\begin{thm} \label{Thm1}
Assume that $A\in \mathbb C^{m\times n}_r$ and $G\in \mathbb C^{n\times m}_s$ is a matrix of rank 
$0<s\leq r$
such that $\ra{GA}=\ra{G}$. Then the sequence of matrix approximations $\{X_k\}_{k=0}^{k=\infty}$ defined by the matrix iteration \eqref{PM} converges to $A_{R{(G)},N{(G)}}^{(2)}$ with the eighteenth order of
convergence if the initial value $X_0=\alpha G$ 
satisfies
\begin{equation}\label{iniccondWG}
\|F_0\|=\|AA_{T,S}^{(2)}-AX_0\|<1.
\end{equation}
Here $\|\cdot\|$ denotes the spectral matrix norm.
\end{thm}
{\bf Proof.} Let us first define the residual matrix in
the $k$th iterate of \eqref{PM} by writing
\begin{equation} \label{f100}
\mathcal{F}_{k}=AA_{T,S}^{(2)}-AX_k.
\end{equation}
Equation (\ref{f100}) can be written as
follows:
\begin{equation}\label{f3000G0a}
\aligned
\mathcal{F}_{k+1}&=AA_{T,S}^{(2)}-AX_{k+1}\\
&=AA_{T,S}^{(2)}-I+I-AX_{k+1}\\
&=AA_{T,S}^{(2)}-I+I-A[X_k((I + R_k)((T_kS_k) + \mu R_k^2 + \psi R_k^4))]\\
&=AA_{T,S}^{(2)}-I+I-A\left[X_k (I+R_k)\left(I+R_k^2+R_k^4+R_k^6+R_k^8+R_k^{10}+R_k^{12}+R_k^{14}+R_k^{16}\right)\right].
\endaligned
\end{equation}
Equation \eqref{f3000G0a}
implies the following relationships:
\begin{equation}\label{f3000G001}
\aligned
\mathcal{F}_{k+1}&=AA_{T,S}^{(2)}-I+(I-AX_{k})^{18}\\
&=AA_{T,S}^{(2)}-I+(I-AA_{T,S}^{(2)}+AA_{T,S}^{(2)}-AX_{k})^{18}\\
&=AA_{T,S}^{(2)}-I+\left[(I-AA_{T,S}^{(2)})+\mathcal{F}_k\right]^{18}.
\endaligned
\end{equation}
Therefore
\begin{equation}\label{f3000G0}
\aligned
\mathcal{F}_{k+1}&=AA_{T,S}^{(2)}-I+[I-AA_{T,S}^{(2)}+\mathcal{F}_k^{18}]\\
&=\mathcal{F}_k^{18}.
\endaligned
\end{equation}
Note that $(I-AA_{T,S}^{(2)})^i=0,\ i>1$, and  that $\mathcal{E}_k=A_{R{(G)},N{(G)}}^{(2)}-X_k$ is the error matrix of the approximation of the outer generalized inverse $A_{T,S}^{(2)}$. Consequently
\begin{equation}\label{formu1}
\aligned
A\mathcal{E}_{k+1}&=AA_{R{(G)},N{(G)}}^{(2)}-AX_{k+1}=\mathcal{F}_{k+1}=\mathcal{F}_k^{18}.
\endaligned
\end{equation}
By using equation \eqref{formu1} and some elementary algebraic transformations,
we 
deduce that
\begin{equation}\label{formu2}
\aligned
\|A\mathcal{E}_{k+1}\|&\leq\|\mathcal{F}_k\|^{18}\\
&=\|A\mathcal{E}_{k}\|^{18}\\
&\leq\|A\|^{18}\|\mathcal{E}_{k}\|^{18}.
\endaligned
\end{equation}
By applying inequality (\ref{formu2})
and assuming that the integer $k$
is large enough, we estimate the rate of convergence as follows:
\begin{equation}\label{formula9}
\aligned
\|\mathcal{E}_{k+1}\|&=\|X_{k+1}-A_{R{(G)},N{(G)}}^{(2)}\|\\
&=\left\|A_{R{(G)},N{(G)}}^{(2)}AX_{k+1}-A_{R{(G)},N{(G)}}^{(2)}AA_{R{(G)},N{(G)}}^{(2)}\right\|\\
&=\left\|A_{R{(G)},N{(G)}}^{(2)}\left(AX_{k+1}-AA_{R{(G)},N{(G)}}^{(2)}\right)\right\|\\
&\leq\|A_{R{(G)},N{(G)}}^{(2)}\|\, \|A\mathcal{E}_{k+1}\|\\
&\leq\|A_{R{(G)},N{(G)}}^{(2)}\|\, \|A\|^{18}\, \|\mathcal{E}_k\|^{18}.
\endaligned
\end{equation}
Therefore
\begin{equation}
\{X_k\}_{k=0}^{k=\infty}\to A_{R{(G)},N{(G)}}^{(2)},
\end{equation}
which shows that the convergence rate is eighteen.  \hfill $\Box$

\section{Stopping Criterion and the Choice of an Initial Approximate Inverse}
According to \cite{Stanimirovic2016},
a reliable stopping criterion for a $p$th order matrix scheme can be expressed as 
follows:
\begin{equation}\label{stopreliable}
\frac{\|X_{k+1}-X_k\|_*}{p^k\alpha}<\epsilon,
\end{equation}
where $\epsilon$ is the tolerance and $\alpha$ is the positive real number
involved in the definition of an initial approximate inverse
$X_0=\alpha G$.

We choose an initial approximation $X_0$ satisfying (\ref{iniccondWG})
in order to ensure convergence. It is sufficient to have this matrix in the form $X_0=\alpha G$
such that
\begin{equation}\label{iniccondWG1}
\left\|A A_{T,S}^{(2)}-AX_0\right\|<1.
\end{equation}
By extending the idea of Pan and Schreiber \cite{Pan1991}, however,
we can choose a more efficient initial value in the form $X_0=\alpha G$, where 
\begin{equation}
\alpha=\frac{2}{\sigma_{1}^2+\sigma_{r}^2},
\end{equation} 
and $\sigma_1\geq \sigma _2\geq \ldots \geq \sigma _r > 0$ are the nonzero eigenvalues of $GA$.

Some initial approximations for  matrices of various types are provided below. For a symmetric positive definite (SPD) matrix $A$, one can apply the Householder-John theorem \cite{Li2011} in order
to obtain  the initial value $X_0=P^{-1},$  where  $P$ can be any matrix such that $P+P^{T}-A$ is SPD. A sub-optimal way of producing $X_0$ for the rectangular matrix $A$ had been given by 
$X_0=\frac{A^*}{\|A\|_{1}\|A\|_{\infty}}$. Also, for finding the Drazin inverse, one may choose  the initial approximation $X_0=\frac{1}{\tr{A^{l+1}}}A^l$ where $\tr{\cdot}$ stands for the trace of a square matrix and $l$ for its {\em index}, $\ind{A}$, that is, the smallest nonnegative integer $l$ such that $rank(A^{l+1})$ $=rank(A^{l})$. The paper  \cite{Gonzalez} proposes some further recipes for the construction of initial inverses specially in the case of square matrices.

\section{Computational Efficiency}
The customary concept of the efficiency index  of  iterative methods can be traced back to 1959 (see \cite{Ehrmann}). Traub in 1964 \cite[Appendix C]{Traub} used this  index in his study of  fixed-point iterations as follows,
\begin{equation}
EI=p^{\frac{1}{c}}.
\end{equation}
Here $c$ stands for the number of dominant cost operations per an iteration loop (in our case they are matrix-by-matrix multiplications), and $p$ denotes the local convergence rate.

Based on the work \cite{SoleymaniLMA}, we estimate that our method (\ref{PM}) converges with the errors within the machine precision in  
approximately
\begin{equation}
s\approx 2log_p\kappa_2(A)
\end{equation}
iteration loops where 
$\kappa_2(A)$ denotes the  condition number of the matrix $A$ in  spectral norm. Recall  that our iteration (\ref{PM}) reaches eighteenth-order convergence by using only seven matrix-by-matrix multiplications. Here are the efficiency indices of various algorithms for generalized inverse:
\begin{equation}\label{efficiencies}
EI_{(\ref{schulz})}=2^{\frac{1}{2}}\approx1.41421,\
EI_{(\ref{chebyshev})}=3^{\frac{1}{3}}\approx1.44225,\
EI_{(\ref{SoleymaniN})}=7^{\frac{1}{5}}\approx1.47577,
\end{equation}
and
\begin{equation}
EI_{(\ref{hyper18-2})}=18^{\frac{1}{9}}\approx1.37872,\
EI_{(\ref{hyper18})}=18^{\frac{1}{18}}\approx1.17419,\
EI_{(\ref{PM})}=18^{\frac{1}{7}}\approx\textbf{1.51121}.
\end{equation}

The latter  efficiency index is  record high for iterative algorithms for generalized inverse.

\section{Numerical Stability Estimates}
In this section we study numerical stability of our iteration (\ref{PM}) in a neighborhood of the solution of the equation 
\begin{equation}
XAX-X=0. 
\end{equation}
We are going to estimate the rounding errors based on the first order error analysis  \cite{Du}.
We recall that iteration (\ref{Hyperpower1})
is self-correcting for computing the 
inverse of a nonsingular matrix, but not so for computing generalized inverses \cite{Soderstorm},
\cite{Pan1991}. 
For the latter task our iteration  (\ref{Hyperpower1}) is  numerically unstable, 
although the instability is rather mild, as we prove next.
In Section 8
 we complement 
our formal study 
 by empirical results.
In our next theorem we proceed under 
the same assumptions as in Theorem \ref{Thm1},
allowing any initial approximation.
In Theorem 6.2 we slightly improve the resulting  estimate 
 assuming the   standard choice of
an initial approximation in the form
$X_0=\alpha A^*$ or more generally $X_0=A^*P(A^*A)$ for 
a constant $\alpha$ and a polynomial  $P(x)$.

\begin{thm}\label{them-stability}
Consider  
 the sequence $\{X_k\}_{k=0}^{k=\infty}$ generated by \eqref{PM}
under the same assumptions as in Theorem \ref{Thm1}.
Write
\begin{equation}\label{stability0}
\widetilde{X}_k=X_k+\Delta_k
\end{equation}
for all $k$ assuming that $\Delta_k$ is a numerical perturbation of the $k$th exact iterate $X_k$
 and has a sufficiently small norm, so that we can ignore quadratic and higher order terms 
in $\mathcal{O}(\Delta_k^2)$.

Then 
\begin{equation}\label{stability2011}
\|\Delta_{k+1}\|\leq \mathcal{C }\|\Delta_0\|
\end{equation}

where
\begin{equation}\label{stability20}
\mathcal{C} = 18^{k+1}\prod_{j=0}^{k}[\max\{1, \|R_j\|^{17}\}(1 + 17\|A\|\|X_j\|)].
\end{equation}

\end{thm}

{\bf Proof.}
Write  $\widetilde{R}_k=I-A\widetilde{X}_k$.
Deduce that for each $j$,
$j=1,...,17$, 
\begin{equation}
\aligned
\|\widetilde{R}_k^j\|
&= \|(R_k - A\Delta_k)^j\|\\
&\leq \|R_k - A\Delta_k\|^j\\
&\leq (\|R_k\| + \|A\Delta_k\|)^j\\
&= C_0^j,
\endaligned
\end{equation}
where $C_0 = \|R_k\| + \|A\Delta_k\| = \|R_k\| + \mathcal{O}(\|\Delta_k\|)$. 
Furthermore

\begin{equation}
\aligned
\|\widetilde{R}_k^j - R_k^j\|
&= \|(R_k - A\Delta_k)^j - R_k^j\|\\
&\leq (\|R_k\| + \|A\Delta_k\|)^j - \|R_k\|^j\\
&= \|A\Delta_k\|(\sum_{i=0}^{j-1}{j \choose j-1-i}\|A\Delta_k\|^i\|R_k\|^{j-1-i})\\
&= D_j\|A\Delta_k\|
\endaligned
\end{equation}
for $D_j = \sum_{i=0}^{j-1}{j \choose j-1-i}\|A\Delta_k\|^i\|R_k\|^{j-1-i} = j\|R_k\|^{j-1} 
+ \mathcal{O}(\|\Delta_k\|)$.


Then deduce that
\begin{equation}
\aligned
\Delta_{k+1} &= \widetilde{X}_{k+1} - \widetilde{X}_k \\
&= \widetilde{X}_k(I + \widetilde{R}_k + \widetilde{R}_k^2 + \cdots + \widetilde{R}_k^{17}) - X_k(I + R_k + R_k^2 + \cdots + R_k^{17})\\
&= (X_k + \Delta_k)(I + (R_k-A\Delta_k) + (R_k-A\Delta_k)^2 + \cdots + (R_k-A\Delta_k)^{17}) - X_k(I + R_k + R_k^2 + \cdots + R_k^{17})\\
&= \Delta_k\sum_{j=0}^{17}[\widetilde{R}_k^j] + X_k\sum_{i=0}^{17}[\widetilde{R}_k^j - R_k^j].\\
\endaligned
\end{equation}

Therefore 
\begin{equation}
\aligned
\|\Delta_{k+1}\| &= \|\Delta_k\sum_{j=0}^{17}[\widetilde{R}_k^j] + X_k\sum_{i=0}^{17}[\widetilde{R}_k^j - R_k^j]\|\\
&\leq \|\Delta_k\|\sum_{j=0}^{17}\|\widetilde{R}_k^j\| + \|X_k\|\sum_{i=0}^{17}\|\widetilde{R}_k^j - R_k^j\|\\
&= \|\Delta_k\|\sum_{j=0}^{17}C_0^j + \|A\Delta_k\|\|X_k\|\sum_{i=0}^{17}D_j\\
&\leq \|\Delta_k\|\sum_{j=0}^{17}[C_0^j + \|A\|\|X_k\|D_j]\\
&< \|\Delta_k\|[18\max\{1, \|R_k\|^{17}\}(1 + 17\|A\|\|X_k\|)] + \mathcal{O}(\|\Delta_k\|).
\endaligned
\end{equation}
 
This yields the claimed estimates (\ref{stability2011}) and (\ref{stability20})
for numerical perturbation at iteration loop $k+1$. 
 \hfill $\Box$\\

The following result a little refines the estimate of Theorem
\ref{them-stability}
 under the standard choices of $X_0$.

\begin{thm}\label{singular-stability}
Consider the same assumptions as in Theorem \ref{Thm1} 
and define singular value decompositions (SVDs) 
\begin{equation}
A = U \left [
\begin{array}{cc}
\Sigma & 0\\
0	& 0
\end{array}
\right ] V^T
\end{equation}
and
\begin{equation}
A^\dagger = V \left [
\begin{array}{cc}
\Sigma^\dagger & 0\\
0	& 0
\end{array}
\right ] U^T
\end{equation}
for a matrix $A$ and its Moore-Penrose pseudo inverse $A^\dagger$.
Moreover, 
let 
$X_0=\alpha A^*$ or more generally let $X_0=A^*P(A^*A)$ for 
a constant $\alpha$ and a polynomial  $P(x)$.
Then write
\begin{equation}
X_k = A^\dagger + E_k = V \left [
\begin{array}{cc}
\Sigma^\dagger + E_{11} & E_{12}\\
E_{21}	& E_{22}
\end{array}
\right ] U^T,
\end{equation}
where $E_k$ is the error of approximation after $k$-th iteration loop.
 Then 
\begin{equation}
\|E_k\| \leq 18^k\|E_0\| + o(\|E_0\|).
\end{equation}
\end{thm}

\begin{proof}
Throughout the proof  drop all terms of second order in $E_k$. Let $R_k = I - AX_k$
and readily deduce that
\begin{equation}
X_kR_k = V\left [
\begin{array}{cc}
-E_{11} & 0 \\
0	& E_{22}
\end{array}
\right ]U^T
\end{equation}
and that 
for $j \geq 2$ ,
\begin{equation}
X_kR_k^j = V\left [
\begin{array}{cc}
0 & 0 \\
0	& E_{22}
\end{array}
\right ]U^T.
\end{equation}
Thus 
\begin{equation}
\aligned
X_{k+1} 
&= X_k(I + R_k + \cdots + R_k^{17})\\
&= V\left [
\begin{array}{cc}
\Sigma^\dagger & E_{12}\\
E_{21} & 18E_{22}
\end{array}
\right ]U^T\\
&\leq A^\dagger + 18\|\Delta_k\|
\endaligned
\end{equation}
and therefore
\begin{equation}
\Delta_k = X_k - A^\dagger \leq 18\|\Delta_{k-1}\| \leq\cdots\leq 18^k\|\Delta_0\|.
\end{equation}
\end{proof}

Two remarks are in order.

\begin{enumerate}
\item
The iteration is not self-correcting,
and so proceeding beyond convergence may seriously increase error and
cause divergence.
\item
Our estimates above do not cover the influence of the rounding errors on
the convergence \cite{SoleymaniAML}. 
The errors may imply slower convergence or even failure of the method,
but this problem is alleviated in the iteration of the next section.
\end{enumerate}

\section{The Most Efficient  Numerically Stable Iteration}
In this section we modify iteration (\ref{Hyperpower1}) by
adding an extra matrix multiplication per iteration loop and then 
prove numerical stability of the modified iteration,
which achieves the 18th order of convergence by 
performing eight
matrix multiplications per iteration loop. Its efficiency index is
$18^{1/8}> 1.435$; this is substantially higher than the previous
record high index among numerically stable iterations
for this task,  equal to  $2^{1/3}<2.26$ (see (7.4) in \cite{Pan1991}).

\begin{thm}\label{Alternative-iteration}
Consider the same assumption as in Theorem \ref{singular-stability} and modify \ref{PM}
 as follows:
\begin{equation}\label{PM-modified}
\left\{ \begin{array}{l}
R_k= I-AX_k,\quad R_k^2 = R_kR_k,\quad R_k^4 = R_k^2R_k^2,\\[1mm]
M_k = (I + c_1 R_k^2 + R_k^4)(I+ c_2 R_k^2 + R_k^4),\\
T_k = M_k + c_3 R_k^2,\quad S_k = M_k + d_1 R_k^2 + d_2 R_k^4,\\
X_{k+1/2} = X_k((I + R_k)((T_kS_k) + \mu R_k^2 + \psi R_k^4))\\
X_{k+1} = X_{k+1/2}AX_{k+1/2}.
\end{array}\right.
\end{equation}
Observe that the modified iteration is numerically stable.
\end{thm}

\begin{proof}
 Assume dealing with the modified procedure and  readily verify that
\begin{equation}
\aligned
X_{k+1/2} 
&= V\left [
\begin{array}{cc}
\Sigma^\dagger & E_{12}\\
E_{21} & 18E_{22}
\end{array}
\right ]U^T
\endaligned
\end{equation}
and 
\begin{equation}
X_{k+1} = X_{k+1/2}AX_{k+1/2} = \left [
\begin{array}{cc}
\Sigma^\dagger & E_{12}\\
E_{21} & 0.
\end{array}
\right ].
\end{equation}
Therefore
\begin{equation}
\Delta_k = X_k - A^\dagger \leq \|\Delta_{k-1}\| \leq\cdots\leq \|\Delta_0\|.
\end{equation}
\end{proof}
One can follow the semi-heuristic recipe of \cite{Pan1991}
by switching from
 iteration  (\ref{Hyperpower1}) 
to iteration (\ref{PM-modified}) 
as soon as all significant singular values
have been suppressed.

\section{Numerical Experiments}
In this section, we present the results of our numerical experiments
for algorithm (\ref{PM}). We applied Mathematica 10.0  \cite[pp. 203-224]{Trott} and carried out our demonstrations with machine precision (except for the first test)
on a computer with the following specifications: Windows 7 Ultimate, Service Pack 1, Intel(R) Core(TM) i5-2430M CPU \@ 2.40GHz, and 8.00 GB of RAM.

For the sake of comparisons, we applied the methods SM, CM, FM, HM and PM, setting the maximum number of iteration loops to 100.
We calculated running time by applying the command $\mathtt{AbsoluteTiming[]}$, which reported the elapsed computational time (in seconds).

We computed the order of convergence in our first  experiment by using
the following expression \cite{SoleymaniLMA},
\begin{equation}\label{COC}
\rho=\frac{\ln\left((\|X_{k+1}-X_{k}\|)(\|X_{k}-X_{k-1}\|)^{-1}\right)}
{\ln\left((\|X_{k}-X_{k-1}\|)(\|X_{k-1}-X_{k-2}\|)\right)^{-1}}.
\end{equation}
Here $\|\cdot\|$ denotes the infinity norm $\|\cdot\|_{\infty}$.

\begin{table}[htb]
\begin{center}
\caption{The results of experiments for Example \ref{ex1}.}
\begin{tabular}{|c|ccccc|}\hline Methods &&SM&CM&FM&PM\\  \hline\hline
$\rho$&&2.00&3.00&7.00&18.00\\
IT&&17&11&7&5 \\
$R_{k+1}$&&$3.712\times10^{-66}$&$1.833\times10^{-59}$&$6.3\times10^{-120}$&$7.474\times10^{-107}$\\\hline
\end{tabular}
\end{center}
\end{table}

\begin{exm}\label{ex1}
In our first series of experiments, we compared various methods for finding the Drazin inverse of the following matrix,
{\small
\begin{equation}\label{f101}
A=\left[
\begin{array}{cccccccccccc}
 2 & 4/10 & 0 & 0 & 0 & 0 & 0 & 0 & 0 & 0 & 0 & 0 \\
 -2 & 4/10 & 0 & 0 & 0 & 0 & 0 & 0 & 0 & 0 & 0 & 0 \\
 -1 & -1 & 1 & -1 & 0 & 0 & 0 & 0 & -1 & 0 & 0 & 0 \\
 -1 & -1 & -1 & 1 & 0 & 0 & 0 & 0 & 0 & 0 & 0 & 0 \\
 0 & 0 & 0 & 0 & 1 & 1 & -1 & -1 & 0 & 0 & -1 & 0 \\
 0 & 0 & 0 & 0 & 1 & 1 & -1 & -1 & 0 & 0 & 0 & 0 \\
 0 & 0 & 0 & -1 & -2 & 4/10 & 0 & 0 & 0 & 0 & 0 & 0 \\
 0 & 0 & 0 & 0 & 2 & 4/10 & 0 & 0 & 0 & 0 & 0 & 0 \\
 0 & -1 & 0 & 0 & 0 & 0 & 0 & 0 & 1 & -1 & -1 & -1 \\
 0 & 0 & 0 & 0 & 0 & 0 & 0 & 0 & -1 & 1 & -1 & -1 \\
 0 & 0 & 0 & 0 & 0 & 0 & 0 & 0 & 0 & 0 & 4/10 & -2 \\
 0 & 0 & 0 & 0 & 0 & 0 & 0 & 0 & 0 & 0 & 4/10 & 2
\end{array}
\right],
\end{equation}}
\noindent where $l=\ind{A}=3$, and we used 150 fixed floating point  digits. The results for this example are given in Table 1, where IT stands for the number of iteration loops
 and $R_{k+1}=\|X_{k+1}-X_k\|_{\infty}$ with the $R_{k+1}\leq \epsilon=10^{-50}$ and $X_0=\frac{1}{\tr{A^{l+1}}}A^l$.
\end{exm}

\begin{exm}\label{ex2}
In this series of experiments, we  compared computational time of various algorithms for computing the Moore-Penrose inverse of 10 rectangular \emph{ill-conditioned Hilbert matrices}
\begin{equation}
H_{m\times n}=\left[\frac{1}{i+j-1}\right]_{m\times n}.
\end{equation}
We used stopping criterion (\ref{stopreliable}) with the Frobenius norm and with the initial approximation $X_0=\frac{2}{\sigma_{1}^2+\sigma_{r}^2}H^T$. The results are displayed in Tables 2-4.
\end{exm}

\begin{table}[htb]
\begin{center}
\begin{tabular}{|c| c c c c|}\hline
 \text{Matrix No.} & \text{SM} & \text{CM} & \text{HM} & \text{PM} \\\hline\hline
 $H_{100\times90}$ & 0.049003 & 0.047003 & 0.048003 & 0.047003 \\
 $H_{200\times190}$ & 0.229013 & 0.220013 & 0.224013 & 0.211012 \\
 $H_{300\times290}$ & 0.435025 & 0.380022 & 0.414024 & 0.388023 \\
 $H_{400\times390}$ & 0.976056 & 0.908052 & 0.920053 & 0.941054 \\
 $H_{500\times490}$ & 1.647094 & 1.531088 & 1.652095 & 1.611092 \\
 $H_{600\times590}$ & 2.595148 & 2.464141 & 2.628150 & 2.505143 \\
 $H_{700\times690}$ & 3.842220 & 3.619207 & 3.956226 & 3.645208 \\
 $H_{800\times790}$ & 5.509315 & 5.131293 & 5.563318 & 5.113292 \\
 $H_{900\times890}$ & 7.238414 & 6.880394 & 7.943454 & 6.872393 \\
 $H_{1000\times990}$ & 9.441540 & 9.019516 & 10.208584 & 9.012523 \\\hline
\end{tabular}
\caption{The elapsed time for Example \ref{ex2} by using $\epsilon=10^{-5}$.}
\end{center}
\end{table}

\begin{table}[htb]
\begin{center}
\begin{tabular}{|c| c c c c|}\hline
 \text{Matrix No.} & \text{SM} & \text{CM} & \text{HM} & \text{PM} \\\hline\hline
 $H_{100\times90}$ & 0.047003 & 0.038002 & 0.047003 & 0.045003 \\
 $H_{200\times190}$ & 0.218012 & 0.194011 & 0.192011 & 0.190011 \\
 $H_{300\times290}$ & 0.387022 & 0.349020 & 0.343020 & 0.337019 \\
 $H_{400\times390}$ & 0.833048 & 0.740042 & 0.809046 & 0.794045 \\
 $H_{500\times490}$ & 1.407080 & 1.373079 & 1.436082 & 1.318075 \\
 $H_{600\times590}$ & 2.273130 & 2.162124 & 2.285131 & 2.067118 \\
 $H_{700\times690}$ & 3.287188 & 3.166181 & 3.351192 & 3.057175 \\
 $H_{800\times790}$ & 4.617264 & 4.405252 & 4.745271 & 4.269244 \\
 $H_{900\times890}$ & 6.221356 & 5.952341 & 6.403366 & 5.784331 \\
 $H_{1000\times990}$ & 8.124465 & 7.841449 & 8.524487 & 7.553432 \\\hline
\end{tabular}
\caption{The elapsed time for Example \ref{ex2} by using $\epsilon=10^{-6}$.}
\end{center}
\end{table}

\begin{table}[htb]
\begin{center}
\begin{tabular}{|c| c c c c|}\hline
 \text{Matrix No.} & \text{SM} & \text{CM} & \text{HM} & \text{PM} \\\hline\hline
 $H_{100\times90}$ & 0.055003 & 0.053003 & 0.048003 & 0.051003 \\
 $H_{200\times190}$ & 0.231013 & 0.233013 & 0.216012 & 0.227013 \\
 $H_{300\times290}$ & 0.463027 & 0.388022 & 0.412024 & 0.406023 \\
 $H_{400\times390}$ & 0.975056 & 0.901052 & 0.933053 & 0.942054 \\
 $H_{500\times490}$ & 1.629093 & 1.568090 & 1.643094 & 1.599092 \\
 $H_{600\times590}$ & 2.651152 & 2.502143 & 2.701154 & 2.471141 \\
 $H_{700\times690}$ & 3.801217 & 3.593205 & 3.992228 & 3.651209 \\
 $H_{800\times790}$ & 5.508315 & 5.130293 & 5.639323 & 5.044289 \\
 $H_{900\times890}$ & 7.192411 & 6.872393 & 7.779445 & 6.890394 \\
 $H_{1000\times990}$ & 9.508544 & 9.211527 & 10.687611 & 9.177525 \\\hline
\end{tabular}
\caption{The elapsed time for Example \ref{ex2} by using $\epsilon=10^{-7}$.}
\end{center}
\end{table}

We compared the efficiency of our iteration (\ref{PM}) and  the known methods. Like the known methods, our iteration converged consistently, but run faster, in good accordance with the formal analysis. Overall the test results in Tables 1-4 confirm some advantages of our iteration in terms of the order of convergence and computational time in most of the tested cases.

\begin{exm}\label{ex3}
Finally we compared the preconditioners obtained from our 
algorithm with the known preconditioners based on Incomplete LU factorizations \cite{Saad} and applied to the solution of the sparse linear systems, $Ax=b$, of the dimension 841 by using GMRES. The matrix $A$ has been chosen from MatrixMarket \cite{http} database as
\begin{equation}
\mathtt{A=ExampleData[{"Matrix","YOUNG1C"}]},
\end{equation}
with the right hand side vector  $b=(1,1,...,1)^T$. The solution in this case is given by
the vector  $(-0.0177027-0.00693171 I$, ..., $-0.0228083-0.00589176 I)^T$. Figure 1 
shows the plot of the matrix $A$ (note that this matrix is not tridiagonal), while Figure 2 reveals the effectiveness of our scheme for preconditioning.
\end{exm}

The left preconditioned system using $X_5$ of SM, $X_3$ of CM, and $X_1$ of PM, along with the well-known preconditioning techniques $\mathtt{ILU0}$, $\mathtt{ILUT}$ and $\mathtt{ILUTP}$ have been tested, while the initial vector has been chosen in all cases automatically, by the command  $\mathtt{LinearSolve[]}$ in Mathematica 10. The results of time comparisons for various values of  tolerance to the
residual norms have been shown in Figure 2.
  In our tests, as could be expected, the computational time increased as  tolerance decreased, but the preconditioner $X_1$ from the method PM mostly yielded the best feedbacks. For these tests, we used the following initial matrix from \cite{Tarazaga},
\begin{equation}\label{choice3}
X_0=diag(1/a_{11},1/a_{22},\cdots,1/a_{nn}),
\end{equation}
where $a_{ii}$ denoted the $i$th diagonal entry of $A$.

After a few iteration loops, the computed preconditioner of the Schulz-type methods
can be dense. Accordingly, we must choose a strategy for controlling the sparseness of the preconditioner. We  can  do this  by setting the Mathematica command $\mathtt{Chop[X,10^{-5}]}$, at the end of each cycle for these matrices.

\begin{figure}[htb]
\centering
\includegraphics[scale=0.7]{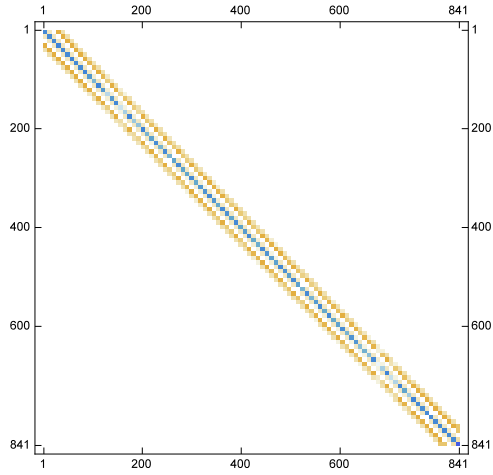}
\caption{The plot of the matrix $A$ in Example \ref{ex3}.}
\end{figure}
\centerline{}
\begin{figure}[htb]
\centering
\includegraphics[scale=0.7]
{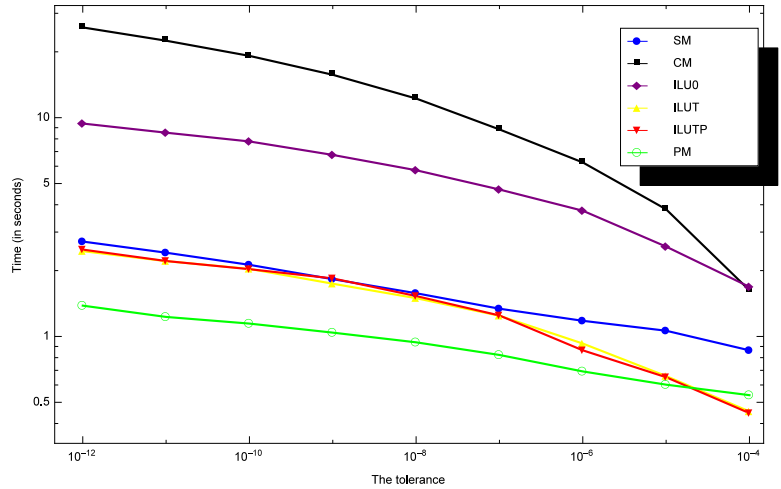}
\caption{The results of comparisons in terms of the computational time.}
\end{figure}

\section{Concluding comments}
The calculation of generalized inverse is an inalienable part of some important matrix computations (see our Section 1.1).

In this paper, we propose a  fast and numerically reliable 
iterative algorithm (\ref{PM}) for the outer generalized inverse $A^{(2)}_{T,S}$ of a matrix $A$. The algorithm has the eighteenth order of convergence and uses only seven matrix-by-matrix multiplications per iteration loop. 
This implies the record high computational efficiency index, 
$\approx 1.511.21$. 
As usual for the iterative algorithms of this class, 
our iteration is self-correcting for computing the 
 inverse of a nonsingular matrix, but not for  computing generalized inverses. 
For that task, the algorithm has mild numerical instability,
but at the expense of performing an extra matrix multiplication per iteration loop,
 we obtain  numerically stable algorithm, still having the eighteenth order of 
convergence. 
This greatly increases the previous record efficiency index, this time in the class 
of numerically stable iterations for generalized inverses.
The  results of our analysis and of   
 our tests indicate that our algorithms are quite promising for practical use in computations with both double and multiple precision. We found out that for high order methods such as (\ref{PM}), it is usually sufficient to perform one full cycle iteration
in order to produce an approximate inverse preconditioner.

Further increase of the convergence order and decrease of the number of matrix multiplications per iteration loop (both under the requirement of numerical stability
and with allowing mild instability) are  natural goals of our future research. 
We are going to extend it also to the acceleration of the known iterative algorithms for  
 various other matrix equations (cf. \cite{Higham2008}, \cite{Bini}).

 \end{document}